\documentclass{amsart}
\usepackage[english]{babel}
\usepackage{amsmath,amssymb,amsthm,amscd}
\usepackage{paralist,prettyref} 
\usepackage{hyperref}
\usepackage{color}

\newtheorem{thm}{Theorem}
\newtheorem{cor}[thm]{Corollary}
\newtheorem{lem}[thm]{Lemma}

\theoremstyle{definition}

\theoremstyle{remark}

\numberwithin{equation}{section}
\newrefformat{cor}{Corollary~\ref{#1}}
\newrefformat{defn}{Definition~\ref{#1}}
\newrefformat{eq}{(\ref{#1})}
\newrefformat{lem}{Lemma~\ref{#1}}
\newrefformat{prob}{problem~(\ref{#1})}
\newrefformat{prop}{Proposition~\ref{#1}}
\newrefformat{thm}{Theorem~\ref{#1}}
\newrefformat{rem}{Remark~\ref{#1}}

\newcommand{\bbC}{\mathbb C}

\newcommand{\bbE}{\mathbb E}
\newcommand{\bbF}{\mathbb F}

\newcommand{\bbN}{\mathbb N}

\newcommand{\bbR}{\mathbb R}
\newcommand{\bbX}{\mathbb X}
\newcommand{\bbY}{\mathbb Y}

\newcommand{\del}{\partial}

\newcommand{\eps}{\varepsilon}

\newcommand{\0}{\mspace{0mu}_0}       

\newcommand{\supp}{\operatorname{supp}} 

\providecommand{\abs}[1]{\left\lvert#1\right\rvert}
\providecommand{\norm}[1]{\lVert#1\rVert}

\begin{document}

\title[Kuznetsov's equation in $L_p$-spaces]{Global well-posedness and exponential stability for Kuznetsov's equation in $L_p$-spaces}


\author[S. Meyer]{Stefan Meyer}
\address{%
Martin-Luther-Universit\"{a}t Halle-Wittenberg\\
Naturwissenschaftliche Fakult\"{a}t II\\
Institut f\"{u}r Mathematik\\
06099 Halle (Saale)}
\email{stefan.meyer@mathematik.uni-halle.de}

\author[M. Wilke]{Mathias Wilke}
\address{%
Martin-Luther-Universit\"{a}t Halle-Wittenberg\\
Naturwissenschaftliche Fakult\"{a}t II\\
Institut f\"{u}r Mathematik\\
06099 Halle (Saale)}
\email{mathias.wilke@mathematik.uni-halle.de}


													

\date{\today}

\begin{abstract}
We investigate a quasilinear initial-boundary value problem for Kuznetsov's equation with non-homogeneous Dirichlet boundary conditions.  This is a model in nonlinear acoustics which describes the propagation of sound in fluidic media with applications in medical ultrasound.  We prove that there exists a unique global solution which depends continuously on the sufficiently small data and that the solution and its temporal derivatives converge at an exponential rate as time tends to infinity.  Compared to the analysis of Kaltenbacher \& Lasiecka, we require optimal regularity conditions on the data and give simplified proofs which are based on maximal $L_p$-regularity for parabolic equations and the implicit function theorem.
\end{abstract}

\maketitle

\section{Introduction}

We are concerned with Kuznetsov's quasilinear equation
\begin{align*}
  u_{tt} - c^2\Delta_x u - b \Delta_x u_t &= k (u^2)_{tt}+\rho_0(v\cdot v)_{tt},
\end{align*}
which is a well-accepted equation in nonlinear acoustics and describes the propagation of sound in fluidic media.  The function $u(t,x)$ denotes the acoustic pressure fluctuation from an ambient value at time $t$ and position $x$.  Furthermore, $c>0$ denotes the velocity of sound, $b>0$ the diffusivity of sound and $k>0$ the parameter of nonlinearity.  The velocity fluctuation $v(t,x)$ is related to the pressure fluctuation by means of an acoustic potential $\psi(t,x)$, such that $u=\rho_0\psi_t$, $v=-\nabla\psi$, with ambient mass density $\rho_0$.

Kuznetsov's equation can be derived from the balances of mass and momentum (the compressible Navier-Stokes equations for a Newtonian fluid) and a state equation for the pressure-dependent density of the fluid, where terms of third or higher order in the fluctuations are neglected.  We refer to the monograph of M. Kaltenbacher \cite{Kal07} and Kuznetsov's article \cite{Kuz71} for the derivation.  Observe that Kuznetsov's equation degenerates if $1-2ku=0$, since $(u^2)_{tt}=2uu_{tt}+2(u_t)^2$, but for $\abs u < (2k)^{-1}$ the equation is parabolic.

Global well-posedness for the corresponding Dirichlet boundary value problem in an $L_2$-setting for the spatial dimension $n\in\{1,2,3\}$ was established by B. Kaltenbacher \& I. Lasiecka \cite{KaLa12} by means of appropriate energy estimates and Banach's fixed point theorem.  The purpose of the present paper is to extend these results to an $L_p$-setting for arbitrary dimensions and to provide shorter and more elegant proofs in an optimal functional analytic setting, in the sense that the regularity conditions on the initial and boundary data are both necessary and sufficient for the regularity of the solution.  We also impose appropriate smallness conditions to avoid the above mentioned degeneracy and to make use of the theory for quasilinear parabolic equations.

Let $\Omega\subset\bbR^n$ be a bounded domain with smooth boundary $\Gamma=\del\Omega$, let $J=(0,T)$ or $J=\bbR_+=(0,\infty)$, and let $u_0,u_1,v_0$ and $g$ be given functions on $\Omega$ and $J\times\Gamma$, respectively.  We consider the following initial-boundary value problem.
\begin{align}\label{prob:Kuznetsov}
   \left\{\begin{aligned}
      u_{tt} - c^2\Delta_x u - b \Delta_x u_t &= k (u^2)_{tt}+\rho_0(v\cdot v)_{tt} &&\quad\text{in }J\times\Omega,\\
      v_t&=-\rho_0^{-1}\nabla u&&\quad\text{in }J\times\Omega,\\
      u|_{\Gamma} &= g &&\quad\text{in }J\times\Gamma,\\
      (u(0),u_t(0)) &= (u_0,u_1) &&\quad\text{in }\Omega,\\
      v(0) &= v_0 &&\quad\text{in }\Omega.
   \end{aligned}\right.
\end{align}
Here $u:J\times\Omega\to\bbR$ is the unknown pressure fluctuation, $(u_0,u_1,v_0):\Omega\to\bbR^{2+n}$ are the given initial data and $g:J\times\Gamma\to\bbR$ is a given inhomogeneity on the boundary.

In this paper we prove existence and uniqueness of strong solutions in suitable subspaces of the Lebesgue space $L_p(J\times\Omega)$ with exponent $p>\max\{1,n/2\}$.  In the case $J=\bbR_+$ we are interested in solutions with exponential decay and therefore consider the functions $t\mapsto e^{\omega t} u(t)$, $t\mapsto e^{\omega t} v_t(t)$ with $\omega\geq 0$, which we abbreviate as $e^{\omega t} u$, $e^{\omega t} v_t$, respectively.  We say that $(u,v)$ is a strong solution to \eqref{prob:Kuznetsov} on $J$, if there is $\omega\geq 0$ such that
\begin{align}\label{eq:RegSol}
  \begin{aligned}
    e^{\omega t}u \in \bbE_u(J) &:= W^2_p(J;L_p(\Omega)) \cap W_p^1(J;W^2_p(\Omega)),\\
    v \in \bbE_v(J) &:= \{ v\in BUC^1(J;W_p^1(\Omega)^n), \\
     &\qquad e^{\omega t}v_t \in H^{3/2}_p(J;L_p(\Omega)^n)\cap W^1_p(J;W_p^1(\Omega)^n))\},
  \end{aligned}
\end{align}
and all equations in \eqref{prob:Kuznetsov} are satisfied pointwise almost everywhere.  Here the symbols $W^s_p$ and $H^s_p$ denote the Sobolev-Slobodeckij and Bessel potential spaces of order $s$ and exponent $p$, respectively, and $BUC^k$ denotes the space of functions having bounded and uniformly continuous derivatives up to order $k$.  The condition $e^{\omega t}u \in \bbE_u(J)$ will be written equivalently as $u \in e^{-\omega}\bbE_u(J)$.  In the case of a finite interval $J=(0,T)$, the exponential factor $e^{\omega t}$ can be dropped and we have $e^{-\omega}\bbE_u(J)=\bbE_u(J)$, for instance.  Our main result is the following.

\begin{thm}\label{thm:GLOBWP}
  Let $n\in\mathbb{N}$, $p>\max\{1,n/2\}$, $p\neq 3/2$, $J=(0,T)$ or $J=\bbR_+$ and define $\omega_0:=\min\{b\lambda_0/2,c^2/b\}>0$, where $\lambda_0>0$ denotes the smallest eigenvalue of the negative Dirichlet-Laplacian in $L_p(\Omega)$.  Then for every $\omega\in(0,\omega_0)$, there is $\rho>0$ such that problem \eqref{prob:Kuznetsov} admits a unique solution
  \begin{align}\label{eq:RegSol_Thm_1}
    (u,v) &\in e^{-\omega}\bbE_u(J) \times \bbE_v(J),
  \end{align}
  if the data $(g,u_0,u_1,v_0)$ satisfy the regularity and compatibility conditions
  \begin{align}\label{eq:RegData}
  \begin{split}
    & u_0 \in W^2_p(\Omega), \quad u_1 \in W_p^{2-2/p}(\Omega), \quad v_0 \in W_p^1(\Omega)^n, \\
     &e^{\omega t}g \in \bbF_g(J) := W^{2-1/2p}_p(J;L_p(\Gamma)) \cap W^1_p(J;W^{2-1/p}_p(\Gamma)), \\
   &\quad g(0)=u_0|_{\Gamma}, \quad \del_t g(0) = u_1|_{\Gamma} \;(\text{if}\ p>3/2),
   \end{split}
   \end{align}
   and the smallness condition
  \begin{align*}
    \norm{g}_{e^{-\omega}\bbF_g} + \|u_0\|_{W^2_p}+\|u_1\|_{W^{2-2/p}_p} + \|v_0\|_{W_p^1} <\rho.
  \end{align*}
  The conditions \eqref{eq:RegData} are necessary for \eqref{eq:RegSol_Thm_1} and the solution depends continuously (even analytically) on the data with respect to the corresponding norms.  In the case $J=\bbR_+$, the solution $(u(t),v(t))$ converges exponentially to $(0,v_\infty)$ as $t\to\infty$, $v_\infty := v_0 - \rho_0^{-1} \int_0^\infty\nabla u(s)ds$, in the sense that
  \begin{align*}
    \norm{u(t)}_{W^2_p} + \norm{u_t(t)}_{W^{2-2/p}_p} + \norm{v(t)-v_\infty}_{W^1_p} + \norm{v_t(t)}_{W^1_p} \leq C e^{-\omega t}, \quad t\geq 0,
  \end{align*}
  for some $C\geq 0$ depending on $g,u_0,u_1,v_0$.
\end{thm}
This theorem extends the results of Kaltenbacher \& Lasiecka \cite[Theorem 1.1 and Theorem 1.2]{KaLa12} in several ways.  First, they only consider the case $p=2$ and $n\in\{1,2,3\}$.  Second, the solution space in the case $J=(0,T)$ is given by
\begin{align}\label{eq:KaLa_solution_space}
  u\in C^2([0,T];L_2(\Omega)) \cap W^2_2(0,T;W^1_2(\Omega)) \cap C([0,T];W^2_2(\Omega))
\end{align}
and the initial data $(u_0,u_1)$ must satisfy the condition $\norm{u_0}_{W^2_2} + \norm{u_1}_{W^2_2} < \rho$
for some (small) $\rho>0$, as a consequence of the assumptions on $(u_0,u_1,g)$.  Most notably, the condition $u_1\in W^2_2(\Omega)$ is not necessary for \eqref{eq:KaLa_solution_space}.  In the case $p=2$ we only require that $\norm{u_0}_{W^2_2} + \norm{u_1}_{W^1_2} < \rho$, where the condition $u_1\in W^1_2(\Omega)$ is necessary, since, by the properties of the temporal trace (see Section 2.2), our solution $u$ must satisfy
\begin{align*}
  u\in C^1([0,T];W^1_2(\Omega)) \cap C([0,T];W^2_2(\Omega)).
\end{align*}
Thirdly, Kaltenbacher \& Lasiecka require the condition
\begin{align*}
  g \in W^3_2(J;W^{-3/2}_2(\Gamma)) \cap W^2_2(J;W^{1/2}_2(\Gamma)) \cap W^1_2(J;W^{3/2}_2(\Gamma)), \, g_{tt}(0)\in W^{-1/2}_2(\Gamma),
\end{align*}
which is also not necessary in view of $u|_\Gamma=g$.  We can simplify and extend it to
\begin{align*}
  g \in W^{7/4}_2(J;L_2(\Gamma)) \cap W^1_2(J;W^{3/2}_2(\Gamma)),
\end{align*}
which is necessary for $u\in\bbE_u(J)$ by the properties of the spatial trace (see Section 2.2).  
Finally, Kaltenbacher \& Lasiecka use that the velocity is given by
\begin{align*}
  v(t,x)=-\frac 1 {\rho_0} \nabla \left(\int_0^t u(s,x)ds+U_0(x)\right), \quad t\geq 0,
\end{align*}
and require that $U_0\in W^2_2(\Omega)\cap \mathring W^1_2(\Omega)$ solves the elliptic problem
\begin{align*}
  -c^2\Delta U_0 - \sigma\nabla u_0\cdot\nabla U_0 &= - (1-2k u_0)u_1 + b\Delta u_0 \text{ in } \Omega, \quad U_0|_\Gamma = 0 \text{ on }\Gamma.
\end{align*}
This implies that $v_0 = -\rho_0^{-1}\nabla U_0$ depends on $u_0, u_1$ and is small in $W^1_2(\Omega)$.  We are able to remove the dependence of $v_0$ on the initial values $u_0$ and $u_1$.

Besides these extensions we point out that in the case $J=\bbR_+$ we always impose some exponential decay on $g$ whereas Kaltenbacher \& Lasiecka impose smallness conditions on the primitive of $g$ instead.

Compared to \cite[Theorem 1.3]{KaLa12}, \prettyref{thm:GLOBWP} does not imply that $\norm{u_{tt}(t)}_{L_p}\to 0$ exponentially as $t\to\infty$, since $\bbE_u(\bbR_+)$ is not contained in $BUC^2(\bbR_+;L_p(\Omega))$.  To obtain such higher order differentiability, we impose additional regularity conditions on $g$ in the following result.
\begin{thm}\label{thm:regularity_and_decay}
Suppose that the conditions of Theorem \ref{thm:GLOBWP} are satisfied for $J=\bbR_+$ and assume in addition that there exists $\omega_g>\omega$ such that 
\begin{align*}
  g, \, [t\mapsto t \del_t g(t)], \,\ldots\,, [t\mapsto t^k \del_t^k g(t)] \in e^{-\omega_g}\bbF_g(\bbR_+) \quad\text{for some }k\geq 1.
\end{align*}
Then the solution $(u,v)$  of \eqref{prob:Kuznetsov} on $\bbR_+$ satisfies
$$(\partial_t^{j}u,\partial_t^{j}v)\in e^{-\omega}\bbE_u([\delta,\infty))\times \bbE_v([\delta,\infty)),$$
for each $\delta>0$ and all $j\in\{0,\ldots,k\}$.  Moreover,
\begin{align*}
  \norm{u(t)}_{W^2_p} + \norm{\del_t u(t)}_{W^2_p} + \cdots + \norm{\del^k_t u(t)}_{W^2_p} + \norm{\del_t^{k+1} u(t)}_{W^{2-2/p}_p} &\leq C e^{-\omega t}, \quad t\geq \delta,\\
  \norm{v(t)-v_\infty}_{W^1_p} + \norm{\del_t v(t)}_{W^1_p} + \ldots + \norm{\del_t^{k+1} v(t)}_{W^1_p} &\leq C e^{-\omega t}, \quad t\geq \delta,
\end{align*}
with some $C\geq 0$ depending on $g,u_0,u_1,v_0$ and $\delta$.
\end{thm}

Finally, we are interested in the case where the inhomogeneity $g$ vanishes except for a finite time interval $(0,T)$.  Then the solution becomes smooth with respect to time on every interval $(T+\delta,\infty)$, $\delta>0$, and all temporal derivatives decay exponentially:
\begin{cor}\label{cor:g_finite_time_interval}
  Suppose that the conditions of \prettyref{thm:GLOBWP} are satisfied for $J=\bbR_+$ and assume in addition that $g(t)=0$ for all $t>T$ with some $T>0$.  Then for every $\omega\in(0,\omega_0)$ there exists $\tilde\rho \leq\rho$ such that for
  \begin{align*}
    \norm{g}_{e^{-\omega}\bbF_g} + \|u_0\|_{W^2_p}+\|u_1\|_{W^{2-2/p}_p} + \|v_0\|_{W_p^1} <\tilde\rho,
  \end{align*}
  the solution $(u,v)$ of \eqref{prob:Kuznetsov} on $\bbR_+$ satisfies
$$(\partial_t^{j}u,\partial_t^{j}v)\in e^{-\omega}\bbE_u([T+\delta,\infty))\times \bbE_v([T+\delta,\infty))$$
for all $j\in\bbN_0$ and every $\delta>0$ and there exist $C_j\geq 0$ such that
\begin{align*}
  \norm{\del^j_t u(t)}_{W^2_p} + \norm{\del_t^j v(t)}_{W^1_p} &\leq C_j e^{-\omega t}, \quad t\geq T+\delta, \, j\in\bbN.
\end{align*}
\end{cor}

This paper is organized as follows. In Section 2 we introduce the notation and we provide some results concerning maximal $L_p$-regularity, trace-theory and analytic mappings in Banach spaces. In Section 3 we study a linearized version of \eqref{prob:Kuznetsov} and prove an optimal regularity result for this linear problem in \prettyref{lem:Max_Reg_for_Linearized_Problem}. The proof of \prettyref{thm:GLOBWP} is given in Section 4. We employ the implicit function theorem combined with the results in Section 3. Finally, in Section 5 we prove  \prettyref{thm:regularity_and_decay} by applying a parameter trick which goes back to Angenent \cite{Ang90} together with the implicit function theorem.

\section{Notation and preliminaries}

In this paper, the symbol $\Omega$ always denotes a bounded domain in $\bbR^n$ with smooth boundary $\Gamma=\del\Omega$ and $J$ denotes the interval $(0,T)$ or $\bbR_+=(0,\infty)$.  $BUC^k(\Omega)$ is the space of $k$-times Fr\'{e}chet-differentiable functions whose derivatives are bounded and uniformly continuous in $\Omega$ and $C^\infty_c(\mathbb{R}^n)$ denotes the space of smooth functions $\phi$ with compact support $\supp \phi\subset \mathbb{R}^n$.   Moreover, $L_p(\Omega)$ denotes the Lebesgue space of exponent $p\in [1,\infty]$, $W^s_p(\Omega)$ denotes the Sobolev-Slobodeckij space of order $s\in[0,\infty)$, $H^s_p(\Omega)$ denotes the Bessel potential space, where we have $W^k_p(\Omega)=H^k_p(\Omega)$ if $k\in\bbN_0 =\bbN\cup\{0\}$, and $W^s_p(\Omega)$ coincides with the Besov space $B^s_{pp}(\Omega)$ if $s\in\bbR_+\setminus\bbN$.  We refer to the monographs \cite{AdFo03,Tri83} for a detailed treatment of these spaces.  We mention that, we will use the Sobolev embeddings $W^1_p(J)\hookrightarrow BUC(J)$, $W^s_p(\Omega)\hookrightarrow BUC^k(\Omega)$ for $s-n/p > k$, $W^s_p(\Omega)\hookrightarrow W^t_q(\Omega)$ for $s-n/p \geq t-n/q$, where the relation $\hookrightarrow$ denotes continuous embedding.  Furthermore, we will use that there exists a bounded linear extension operator from $W^s_p(\Omega)$ to $W^s_p(\bbR^n)$.


\subsection{Maximal regularity}
By the work of Amann \cite{Ama93,Ama95}, Denk, Hieber \& Pr\"uss \cite{DHP03,DHP07} and many others, it is nowadays well-known that the nonhomogeneous initial-boundary value problem for the heat equation,
\begin{align*}
    u_t - \Delta u &= f \text{ in }J\times\Omega, \quad u|_\Gamma = g\text{ on }J\times\Gamma, \quad u|_{t=0}=u_0 \text{ in }\Omega,
\end{align*}
has \emph{maximal $L_p$-regularity} in the sense that there exists a unique solution
\begin{align}\label{eq:E_1}
  u \in \bbE_1(J) := W^1_p(J;L_p(\Omega))\cap L_p(J;W^2_p(\Omega)),
\end{align}
if and only if the given data $f,g,u_0$ satisfy the regularity conditions
\begin{align*}
  &f\in L_p(J\times\Omega), \quad g\in W^{1-1/2p}_p(J;L_p(\Gamma))\cap L_p(J;W^{2-1/p}_p(\Gamma)), \quad u_0\in W^{2-2/p}_p(\Omega),
\end{align*}
and the compatibility condition $u_0|_\Gamma = g|_{t=0}$ in the sense of traces.  By Banach's closed graph theorem, this implies that a solution satisfies the a priori estimate
\begin{align*}
  \norm u_{\bbE_1(J)} \lesssim \norm f_{L_p(J\times\Omega)} + \norm g_{\bbF_\Gamma(J)} + \norm{u_0}_{W^{2-2/p}_p(\Omega)}.
\end{align*}
%
%
Here the symbol $\lesssim$ indicates that the left-hand side $\norm u_{\bbE_1(J)}$ can be estimated by a constant $C\geq 0$ times the right-hand side, where $C$ does not depend on $u, f, g, u_0$.


\subsection{Trace theory}

To verify the necessity of the conditions on the data, we use Lemma 3.5 and Lemma 3.7 in \cite{DHP07}, which imply that the spatial and temporal traces
\begin{align}
  u \mapsto u|_{\Gamma} &: \bbE_1(J) \to \bbF_\Gamma(J) := W^{1-1/2p}_p(J;L_p(\Gamma))\cap L_p(J;W^{2-1/p}_p(\Gamma)),\label{eq:F_Gamma}\\
  u \mapsto u|_{t=0} &: \bbE_1(J) \to W^{2-2/p}_p(\Omega)\nonumber
\end{align}
are bounded and surjective for every $p\in(1,\infty)$.  The temporal trace theorem is also known in the following form \cite[Theorem III.4.10.2]{Ama95}:
\begin{align*}
  \bbE_1(J) \hookrightarrow BUC(J;W^{2-2/p}_p(\Omega)).
\end{align*}
For the compatibility conditions we use that the temporal trace
\begin{align*}
  g \mapsto g|_{t=0} &: \bbF_\Gamma(J) \to B^{2-3/p}_{pp}(\Gamma)
\end{align*}
is well-defined and bounded in the case $p>3/2$ \cite[Lemma 11]{diB84}.  Moreover, the spatial trace $W^s_p(\Omega)\to B^{s-1/p}_{pp}(\Gamma)$ is bounded for $s\in(1/p,\infty)$ \cite[Theorem 3.3.3]{Tri83}. Therefore we obtain $g(0)=u_0|_\Gamma$ in the sense of $B^{2-3/p}_{pp}(\Gamma)$ for $p>3/2$.  In the case $p< 3/2$ these traces do not exist (see e.g.\ \cite{DHP07}).  The case $p=3/2$ is excluded, since the trace space looks more complicated in this case \cite[Theorem 4.3.3]{Tri95}.

\subsection{Exponentially weighted spaces}

Let $Y$, $\bbX(J)$ be Banach spaces such that $\bbX(J)\hookrightarrow L_{1,\text{loc}}(J;Y)$ and let $\omega\in\bbR$. To describe exponential decay of solutions we employ the Banach space
\begin{align*}
  e^{-\omega}\bbX(J) := \{u\in L_{1,\text{loc}}(J;Y): (t\mapsto e^{\omega t}u(t)) \in \bbX(J)\},
\end{align*}
equipped with the norm $\|u\|_{e^{-\omega}\bbX(J)} := \|e^{\omega \cdot} u\|_{\bbX(J)}$, where we write $e^{\omega \cdot} u$ or $e^{\omega t} u$ instead of $(t\mapsto e^{\omega t}u(t))$ for the sake of brevity.

%
%

\subsection{Analytic mappings between Banach spaces}

Let $X$, $Y$ be Banach spaces over the same scalar field $\bbR$ or $\bbC$ and let $U\subset X$ be open.  We say that a mapping $F:U\subset X\to Y$ is \emph{analytic}\index{analytic operator} at $u\in U$, if there exists $r>0$ and bounded symmetric $k$-linear operators $F_{k} : X^k = X \times\cdots\times X \to Y$, $k\geq 0$, such that every $F(u+h)$ for $h\in X$, $\norm h<r$, can be represented as
\begin{align}\label{eq:analytic_power_series}
  F(u+h) = \sum_{k= 0}^\infty F_{k} h^k, \quad \text{such that} \quad \sum_{k= 0}^\infty \norm{F_{k}} \norm h^k < \infty,
\end{align}
where $\norm{F_k}$ denotes the norm of the $k$-linear operator $F_k$, i. e. the smallest number $M\geq 0$ such that $\norm {F_k(x_1,\ldots,x_k)}_Y\leq M \norm{x_1}_X\cdots\norm{x_k}_X$ for all $x_1,\ldots,x_k\in X$.  We refer to the monographs \cite{Dei85,Zei86I} for more information on analytic mappings.

\section{The linearized problem}

In this section we establish maximal regularity for the linearization of problem \eqref{prob:Kuznetsov} in the following sense.

\begin{lem}\label{lem:Max_Reg_for_Linearized_Problem}
Let $1<p<\infty$, $p\neq 3/2$, $n\in \mathbb{N}$, let $\lambda_0>0$ denote the smallest eigenvalue of the negative Dirichlet-Laplacian in $L_p(\Omega)$ and define $\omega_0 := \min\{b\lambda_0/2,c^2/b\}$.  For every $\omega\in (0,\omega_0)$ the following assertion holds.  There exists a unique solution
$$u\in e^{-\omega}\bbE_u, \quad \bbE_u := W_p^2(\mathbb{R}_+;L_p(\Omega))\cap W_p^1(\mathbb{R}_+;W_p^2(\Omega))$$
of the linear problem
\begin{align}\label{prob:Kuznetsov_Linearized}
    \left\{\begin{aligned}
        u_{tt}-c^2\Delta u-b\Delta u_t &= f &&\text{ in } \bbR_+\times\Omega,\\
        u|_{\Gamma} &= g &&\text{ on } \bbR_+\times\Gamma,\\
        u(0) &= u_0 &&\text{ in } \Omega, \\
        u_t(0) &= u_1 &&\text{ in } \Omega,
    \end{aligned}\right.
\end{align}
if and only if the data satisfy the following conditions:
\begin{enumerate}
\item $f\in e^{-\omega}L_p(\mathbb{R}_+\times\Omega)$;
\item $g\in e^{-\omega}\bbF_g$, $\bbF_g:=W_p^{2-1/2p}(\mathbb{R}_+;L_p(\Gamma))\cap W_p^1(\mathbb{R}_+;W_p^{2-1/p}(\Gamma))$;
\item $u_0\in W_p^2(\Omega)$, $u_1\in W_p^{2-2/p}(\Omega)$;
\item $g(0)=u_0|_\Gamma$ and in the case $p\in (3/2,\infty)$ also $g_t(0)=u_1|_{\Gamma}$.
\end{enumerate}
The solution satisfies the estimate
$$\|u\|_{e^{-\omega}\mathbb{E}_u}\lesssim \|f\|_{e^{-\omega} L_p(\mathbb{R}_+\times\Omega)}+\|g\|_{e^{-\omega}\mathbb{F}_g}+\|u_0\|_{W_p^2}
+\|u_1\|_{W_p^{2-2/p}}.$$
\end{lem}

In order to prove this result, we will use the subsequent lemma to remove the inhomogeneity $g$.  The symbols $\bbE_1 := \bbE_1(\bbR_+)$, $\bbF_\Gamma := \bbF_\Gamma(\bbR_+)$ are defined in \eqref{eq:E_1} and \eqref{eq:F_Gamma}, respectively.

\begin{lem}\label{lem:Inhom_bdy_data}
  Let $1<p<\infty$, $p\neq 3/2$, $n\in \mathbb{N}$ and let $\lambda_0>0$ denote the smallest eigenvalue of the negative Dirichlet-Laplacian in $L_p(\Omega)$.  Then for every $\omega\in(0,b\lambda_0)$, there exists a unique solution $w\in e^{-\omega}\bbE_u$ of
  \begin{align}\label{prob:Inhom_bdy_data}
     \left\{\begin{aligned}
        w_{tt} - b\Delta w_t &= 0 &&\text{ in } \bbR_+\times\Omega,\\
        w|_\Gamma &= g &&\text{ on } \bbR_+\times\Gamma,\\
        w|_{t=0} &= 0 &&\text{ in } \Omega, \\
        \del_t w|_{t=0} &= 0 &&\text{ in } \Omega,
     \end{aligned}\right.
  \end{align}
  if and only if $g\in e^{-\omega} \bbF_g$, $g(0)=0$, and in the case $p>3/2$ also $g_t(0)=0$.  The solution $w$ satisfies the estimate $\norm{e^{\omega} w}_{\bbE_u} \lesssim \norm{e^\omega g}_{\bbF_g}$.
\end{lem}
\begin{proof}
  We first prove sufficiency.  Using maximal regularity \cite[Proposition 8 and formula (52)]{LatPrSchn06}, we obtain a unique solution $v\in e^{-\omega}\bbE_1$ of the parabolic problem
  \begin{align*}
    v_{t} - b\Delta v = 0 \text{ in } \bbR_+\times\Omega, \quad v|_\Gamma = \del_t g \text{ on } \bbR_+\times\Gamma, \quad v|_{t=0}=0 \text{ in } \Omega.
  \end{align*}
  Indeed, the operator $\del_t : \bbF_g\to\bbF_\Gamma$ is bounded and thus $\norm{e^{\omega t}\del_t g}_{\bbF_\Gamma} = \norm{\del_t(e^{\omega  t}g)-\omega e^{\omega t} g}_{\bbF_\Gamma} \lesssim \norm{e^{\omega t}g}_{\bbF_g}$.  Next we define
  \begin{align*}
    w(t,x) := - \int_t^\infty v(s,x)ds, \quad t\in \bbR_+, \; x\in\Omega.
  \end{align*}
  Using $\omega>0$, we infer from $w_t=v$, from the identity
  $$e^{\omega t}w(t):=-\int_t^\infty e^{\omega (t-s)}e^{\omega s}v(s)\ ds = (e^{-\omega t}\chi_{\bbR_+})\ast(e^{\omega t} v)(t)$$
  and Young's inequality, that $w\in e^{-\omega}\bbE_u$.  Moreover it holds that
  $$w_t(t)=v(t)=-\int_t^\infty v_s(s)\ ds=-b\int_t^\infty \Delta v(s)\ ds=b\Delta w(t),$$
  $$w(t)|_\Gamma=-\int_t^\infty g_s(s)\ ds=g(t)$$
  and $w_t(0)=v(0)=0$. It follows that $b\Delta w(0)=w_t(0)=0$ and $w(0)|_{\Gamma}=g(0)=0$, hence $w(0)=0$ in $\Omega$.  Therefore $w$ is a solution to \eqref{prob:Inhom_bdy_data}.

  The necessity follows from the spacial trace theorem applied to $w, w_t \in e^{-\omega}\bbE_1$.  To obtain uniqueness, let $w$ be a solution to \eqref{prob:Inhom_bdy_data} with $g=0$.  Since $w_t$ solves a heat problem with homogeneous data, we obtain $w_t=0$ and therefore also $w=0$ by the initial condition $w(0)=0$.  The estimate follows from the closed graph theorem.
\end{proof}

\begin{proof}[Proof of \prettyref{lem:Max_Reg_for_Linearized_Problem}.]
We obtain uniqueness from our previous maximal regularity result \cite[Theorem 2.5]{MeyWi11} for \eqref{prob:Kuznetsov_Linearized} for the case $g=0$. To verify the necessity of the conditions on the data, suppose that $u$ is a solution to \eqref{prob:Kuznetsov_Linearized}.  Then $e^{\omega t}u$ and $(e^{\omega t}u)_t = \omega e^{\omega t} u + e^{\omega t}u_t$ belong to $\bbE_1$ and (i) is readily checked.  Taking the spatial trace yields $e^{\omega t}g$, $(e^{\omega t}g)_t \in \bbF_\Gamma$ which implies (ii).  The exponential weight $e^{\omega t}$ does not affect the initial regularity and therefore (iii) follows by taking the temporal trace and using the embedding $W^1_p(J;W^2_p(\Omega))\hookrightarrow BUC(J;W^2_p(\Omega))$.  Using that \begin{align*}
  (g,g_t)\in BUC(J;W^{2-1/p}_p(\Gamma)\times B_{pp}^{2-3/p}(\Gamma))
\end{align*}
and applying the spatial trace to $u(0)$, $u_t(0)$ and the temporal trace to $(g,g_t)$, we see that $g(0)=u_0|_\Gamma$ is valid in the sense of $W^{2-1/p}_p(\Gamma)$ for all $p$ and $\del_t g(0) = u_1|_\Gamma$ is valid in the sense of $B^{2-3/p}_{pp}(\Gamma)$ if $p>3/2$.

It remains to prove sufficiency of the conditions.  First we reduce the problem to the case $u_0=0$, $u_1=0$, $f=0$.  This cannot be done by just solving the problem with $g=0$, due to the compatibility conditions.  Therefore we extend $u_0$, $u_1$ and $f$ to some functions $\tilde{u}_0\in W_p^2(\mathbb{R}^n)$, $\tilde{u}_1\in W_p^{2-2/p}(\mathbb{R}^n)$ and $\tilde{f}\in e^{-\omega}L_p(\mathbb{R}_+\times\mathbb{R}^n)$.  By means of a cut-off function $\phi\in C_c^\infty(\mathbb{R}^n)$ such that $\phi(x)=1$ for $x\in\Omega$ and $\phi(x)=0$ for $x\notin B_R :=\{y\in\bbR^n : \abs y < R\}$ for some $R>0$, we define new data $\hat{u}_0:=\tilde{u}_0\phi$, $\hat{u}_1:=\tilde{u}_1\phi$ and $\hat{f}:=\tilde{f}\phi$ and consider the problem
\begin{align*}
    \left\{\begin{aligned}
        \hat{u}_{tt}-\Delta \hat{u}-\Delta \hat{u}_t&=\hat{f}, &&\quad t>0,\ x\in B_R,\\
        \hat{u}|_{\partial B_R}&=0, &&\quad t>0,\ x\in \partial B_R,\\
        \hat{u}(0)=\hat{u}_0,\ \hat{u}_t(0)&=\tilde{u}_1, &&\quad t=0,\ x\in B_R.
    \end{aligned}\right.
\end{align*}
Using maximal regularity \cite[Theorem 2.5]{MeyWi11}, we obtain a unique solution
\begin{align*}
\hat{u}\in e^{-\omega}[W_p^2(\mathbb{R}_+;L_p(B_R))\cap W_p^1(\mathbb{R}_+;W_p^2(B_R))].
\end{align*}
Let $\bar{u}$ denote the restriction of $\hat{u}$ to $\Omega$ and let $\bar{g}:=g-\bar{u}|_{\Gamma}$.  Then the final solution $u$ will be given by $u=v+\bar u$, where $v$ solves the problem
\begin{align*}
  \left\{\begin{aligned}
    v_{tt}-\Delta v-\Delta v_t&=0, &&\quad t>0,\ x\in \Omega,\\
    v|_{\Gamma}&=\bar{g}, &&\quad t>0,\ x\in \Gamma,\\
    v(0)=0,\ v_t(0)&=0, &&\quad t=0,\ x\in\Omega.
  \end{aligned}\right.
\end{align*}
From \prettyref{lem:Inhom_bdy_data} we obtain a unique solution $\bar v \in e^{-\omega}\bbE_u$ of the problem
\begin{align*}
  \left\{\begin{aligned}
    \bar v_{tt} - b\Delta \bar v_t &= 0 &&\quad\text{ in } \bbR_+\times\Omega,\\
    \bar v|_\Gamma &= \bar{g} &&\quad\text{ on } \bbR_+\times\Gamma,\\
    \bar v|_{t=0} &= 0 &&\quad\text{ in } \Omega, \\
    \del_t \bar v|_{t=0} &= 0 &&\quad\text{ in } \Omega.
  \end{aligned}\right.
\end{align*}
Then the function $w:=v-\bar{v}$ solves the problem
\begin{align*}
  \left\{\begin{aligned}
    w_{tt}-\Delta w-\Delta w_t &= \Delta \bar{v} &&\quad \text{ in } \bbR_+\times\Omega,\\
    w|_{\Gamma} &= 0 &&\quad \text{ on } \bbR_+\times\Gamma,\\
    w(0)=0, \ w_t(0) &= 0 &&\quad \text{ in } \Omega,
  \end{aligned}\right.
\end{align*}
which has a unique solution $w\in e^{-\omega}\bbE_u$
by \cite[Theorem 2.5]{MeyWi11}. The function $u:= w+\bar{v}+\bar{u}$ is the desired solution of \eqref{prob:Kuznetsov_Linearized} and the estimate follows from the closed graph theorem. This concludes the proof of \prettyref{lem:Max_Reg_for_Linearized_Problem}.
\end{proof}

\section{The nonlinear problem}

In this section we construct a solution to the nonlinear problem \eqref{prob:Kuznetsov} of the form $(u+u_*,v)$.  Here $u_*$ solves the linearized problem \eqref{prob:Kuznetsov_Linearized} for the data $(f=0,g,u_0,u_1)$ and $u$ satisfies homogeneous boundary and initial conditions.  The (small) deviation $(u,v)$ from $(u_*,0)$ will be found by the implicit function theorem.

For $p>\max\{1,n/2\}$, we employ the Banach function spaces
\begin{align}\label{eq:RegSol_2}
\begin{aligned}
  \bbE_{u} &:= W_p^2(\mathbb{R}_+;L_p(\Omega)) \cap  W_p^1(\mathbb{R}_+;W^2_p(\Omega)), \\
  \0 \bbE_{u,\Gamma} &:= \{u\in\bbE_u : u(0)=u_t(0)=0, \, u|_\Gamma = 0\},\\
  \bbE_v &:= \{v\in BUC^1(\mathbb{R}_+;W_p^1(\Omega)^n):(t\mapsto e^{\omega t}v_t)\in BUC(\mathbb{R}_+;W_p^1(\Omega)^n)\}, \\
  \bbE_{v_t} &:= BUC(\mathbb{R}_+;W_p^1(\Omega)^n),
\end{aligned}
\end{align}
Observe that now $\bbE_v$ is a somewhat larger space compared to \eqref{eq:RegSol}.

\begin{lem}\label{lem:Nonlinearity_is_analytic}
  Let $\Omega\subset\bbR^n$ be a bounded domain with boundary $\Gamma=\del\Omega\in C^2$ and let
  $p>\max\{1,n/2\}$, $p\neq 3/2$.  Moreover, let $\omega\in(0,\omega_0)$ have the same meaning as in \prettyref{lem:Max_Reg_for_Linearized_Problem}.  Then the mapping
  $$H:e^{-\omega}\0\bbE_{u,\Gamma}\times \bbE_v\times e^{-\omega}\bbE_u\times W_p^1(\Omega)^n\to e^{-\omega} L_p(\mathbb{R}_+\times\Omega)\times e^{-\omega} \bbE_{v_t} \times W_p^1(\Omega)^n,$$
  defined by
  \begin{align*}
    &H(u,v,u_*,v_0)\\&:=\begin{pmatrix}u_{tt} - c^2\Delta u-b\Delta u_t-k((u+u_*)^2)_{tt}-2\rho_0^{-1}(\nabla (u+u_*))^2-2v\cdot \nabla(u+u_*)_t\\
  v_t+\rho_0^{-1}\nabla(u+u_*)\\
  v(0)-v_0\end{pmatrix}
  \end{align*}
  is analytic and its Fr\'{e}chet derivative w. r. t. $(u,v)$ at $(0,0,0,0)$ is invertible. 
\end{lem}

\begin{proof}
  We define the space $\0 W^1_p(J;X) := \{u\in W^1_p(J;X) : u(0)=0 \}$.  It is easy to check that the linear operators
  \begin{align*}
    u\mapsto u_{tt} &: e^{-\omega}\0\bbE_{u,\Gamma}\to e^{-\omega}L_p(\bbR_+\times\Omega), \\
    u\mapsto -\Delta u &: e^{-\omega}\0\bbE_{u,\Gamma}\to e^{-\omega}\0 W^1_p(\bbR_+;L_p(\Omega)), \\
    u\mapsto -\Delta u_t &: e^{-\omega}\0\bbE_{u,\Gamma}\to e^{-\omega}L_p(\bbR_+\times\Omega), \\
    v\mapsto v_t &: \bbE_v \to e^{-\omega}\bbE_{v_t}, \\
    v\mapsto v_0 &: \bbE_v \to W^1_p(\Omega)^n
  \end{align*}
  are bounded and analytic.  Next, we check that
  \begin{align}\label{eq:nonlinearity_1}
    (u,u_*)\mapsto ((u+u_*)^2)_{tt} &: e^{-\omega}\0\bbE_{u,\Gamma}\times e^{-\omega}\bbE_{u}\to e^{-\omega}L_p(\bbR_+\times\Omega)
  \end{align}
  is analytic.  From the preliminaries, we obtain the continuity of the embeddings
  \begin{align}
    \bbE_u &\hookrightarrow BUC(\bbR_+;W^2_p(\Omega)) \nonumber\\
    &\hookrightarrow BUC(J\times \Omega) && \text{(valid for $p>n/2$)},\nonumber\\
    \bbE_u &\hookrightarrow H^{1+\theta-\eps/2}_p(\bbR_+;H^{2-2\theta+\eps}_p(\Omega)) && \text{(valid for $\theta-\eps/2\in[0,1]$, $\eps>0$)}\label{eq:E_u_to_mixed_derivative_space}\\
    &\hookrightarrow W^{1+\theta-\eps}_p(\bbR_+;W^{2-2\theta}_p(\Omega)) && \text{(valid for $\eps>0$)}\nonumber\\
    &\hookrightarrow W^{1}_{2p}(\bbR_+;W^{2-2\theta}_p(\Omega)) &&\text{(valid for $\theta\geq 1/2p+\eps$)}\nonumber\\
    &\hookrightarrow W^{1}_{2p}(\bbR_+;L_{2p}(\Omega)) &&\text{(valid for $\theta\leq 1-n/4p$)}\nonumber.
  \end{align}
  Here, \eqref{eq:E_u_to_mixed_derivative_space} is a consequence of Sobolevskij's mixed derivative theorem \cite[Proposition 3.2]{DHP07}, \cite[Lemma 4.1]{DSS08} and it is a special case of \cite[Proposition 3.2]{MeSc12}.  Such numbers $\theta$, $\eps$ exist if $p\geq n/4+1/2+\eps$, which is true if $p>\max\{1,n/2\}$ and provided that $\varepsilon>0$ is sufficiently small.  Furthermore, with the norm $\norm\cdot_p$ of $L_p(\bbR_+\times\Omega)$, the estimates
  \begin{align*}
    \norm {fg}_p &\leq \norm f_{2p} \norm g_{2p} \lesssim \norm f_{\bbE_u}\norm g_{\bbE_u}, \\
    \norm{(fg)_t}_p &\leq \norm {f_t}_{2p} \norm g_{2p} + \norm f_{2p}\norm{g_t}_{2p} \lesssim \norm f_{\bbE_u}\norm g_{\bbE_u}, \\
    \norm{(fg)_{tt}}_p &\leq \norm{f_{tt}}_p\norm g_\infty + 2 \norm{f_t}_{2p}\norm{g_t}_{2p} + \norm{f}_\infty\norm{g_{tt}}_{p} \lesssim \norm f_{\bbE_u}\norm g_{\bbE_u}
  \end{align*}
  imply that $(f,g)\mapsto fg : \bbE_u \times \bbE_u \to W^2_p(\bbR_+;L_p(\Omega))$ is bilinear, symmetric and bounded.  Using that $\omega \geq 0$ and $e^{\omega t}\leq e^{2\omega t}$, we see that $e^{-2\omega}L_p(\bbR_+\times\Omega) \hookrightarrow e^{-\omega}L_p(\bbR_+\times\Omega)$.  Setting $f=e^\omega u$, $g=e^\omega u_*$, $w=u+u_*$ and using
  \begin{align*}
    e^{2\omega t}(w^2)_{tt} &= ((e^{\omega t}w)^2)_{tt} -4\omega ((e^{\omega t}w)^2)_{t} +4\omega^2 (e^{\omega t}w)^2,
  \end{align*}
  we conclude that \eqref{eq:nonlinearity_1} is quadratic and continuous and thus analytic.

  In a similar way, we can check that
  \begin{align*}
    (u,u_*)\mapsto (\nabla u+\nabla u_*)^2 &: e^{-\omega}\0\bbE_{u,\Gamma}\times e^{-\omega}\bbE_{u}\to e^{-\omega}L_p(\bbR_+\times\Omega)
  \end{align*}
  is analytic, where we make use of the continuity of the embeddings
  \begin{align*}
    \bbE_u \hookrightarrow W^{1}_p(\bbR_+;W^{2}_p(\Omega)) &\hookrightarrow W^1_p(\bbR_+;W^1_{2p}(\Omega)) && \text{(valid for $p\geq n/2$)}\\
    &\hookrightarrow L_{2p}(\bbR_+;W^1_{2p}(\Omega)) &&\text{(valid for $p\geq 1/2$)}
  \end{align*}
  and the inequality
  \begin{align*}
    \norm{\nabla(fg)}_p \leq \norm {\nabla f}_{2p} \norm g_{2p} &+ \norm f_{2p}\norm{\nabla g}_{2p} \lesssim \norm f_{\bbE_u}\norm g_{\bbE_u}.
  \end{align*}

  To obtain the analyticity of
  \begin{align*}
    (u,v,u_*)\mapsto v\cdot(\nabla u+\nabla u_*)_t &: e^{-\omega}\0\bbE_{u,\Gamma} \times \bbE_v \times e^{-\omega}\bbE_{u}\to e^{-\omega}L_p(\bbR_+\times\Omega),
  \end{align*}
  we use $ e^{\omega t}\nabla w_t = \nabla( e^{\omega t}w)_t-\omega e^{\omega t}\nabla w$ and the estimate
  \begin{align*}
    \norm{v\cdot e^{\omega t}\nabla w_t}_p \leq \norm v_{L_\infty(\bbR_+;L_{2p}(\Omega))}\norm {e^{\omega t}\nabla w_t}_{L_p(\bbR_+;L_{2p}(\Omega))} \lesssim \norm v_{\bbE_v}\norm { e^{\omega t}w}_{\bbE_u},
  \end{align*}
  which is valid for $p\geq n/2$ since $W^1_p(\Omega)\hookrightarrow L_{2p}(\Omega)$ is continuous in this case.

%
The Fr\'{e}chet derivative of $H$ w. r. t. $(u,v)$ at $(0,0,0,0)$ is given by
$$D_{(u,v)}H(0,0,0,0)[\bar{u},\bar{v}] = \begin{pmatrix}\bar{u}_{tt} - c^2\Delta \bar{u}-b\Delta\bar{u}_t\\
\bar{v}_t+\rho_0^{-1}\nabla\bar{u}\\
\bar{v}(0)\end{pmatrix}.$$
We will now show that $D_{(u,v)}H(0,0,0,0) : \bbE\to\bbF$ is an isomorphism, where
\begin{align*}
  \bbE := e^{-\omega}\0\bbE_{u,\Gamma} \times \bbE_v, \quad \bbF := e^{-\omega} L_p(\mathbb{R}_+\times\Omega)\times e^{-\omega}\bbE_{v_t}\times W_p^1(\Omega)^n.
\end{align*}
To this end let $f=(f_1,f_2,f_3)\in \bbF$ and consider the system $$D_{(u,v)}H(0,0,0,0)[\bar{u},\bar{v}]=f.$$ By \cite[Theorem 2.5]{MeyWi11} there exists a unique solution $\bar{u}\in e^{-\omega}\0\bbE_{u,\Gamma}$ of the first equation. Inserting this solution into the second equation yields $v_t(t)=-\rho_0^{-1}\nabla\bar{u}(t)+f_2(t)$.
Integrating w. r. t. $t$ and invoking the initial condition $\bar{v}(0)=f_3$, we obtain
$$\bar{v}(t)=-\int_0^t \rho_0^{-1}\nabla\bar{u}(s) ds+\int_0^tf_2(s) ds+f_3.$$
This function belongs to $\bbE_v$, as can be seen from
\begin{align*}
  \nabla\bar{u} &\in e^{-\omega}W_p^1(\mathbb{R}_+;W_p^1(\Omega)^n) \hookrightarrow e^{-\omega} BUC(\mathbb{R}_+;W_p^1(\Omega)^n)=e^{-\omega}\mathbb{E}_{v_t}.\qedhere
\end{align*}
\end{proof}

\begin{proof}[Proof of \prettyref{thm:GLOBWP}]
It suffices to consider the case $J=\bbR_+$, since the considered function spaces over $J=(0,T)$ can be identified with subspaces of the corresponding spaces over $\bbR_+$ by means of extension and restriction, see \cite[Theorem 5.19]{AdFo03} for the scalar-valued case and \cite[Lemma 2.5]{MeSc12} for the vector-valued case.

As a consequence of \prettyref{lem:Nonlinearity_is_analytic} and since $H(0,0,0,0)=(0,0,0)$,  the implicit function theorem yields a (possibly small) ball $B_\rho(0)\subset e^{-\omega}\bbE_u\times W_p^1(\Omega)^n$ and an analytic mapping $$\psi:B_\rho(0)\subset e^{-\omega}\bbE_u\times W_p^1(\Omega)^n\to e^{-\omega}\0\bbE_{u,\Gamma}\times\bbE_v, \quad (u_*,v_0) \mapsto (u,v)=\psi(u_*,v_0)$$ with $\psi(0,0)=(0,0)$ such that
\begin{align*}
  H(\psi(u_*,v_0),(u_*,v_0))=(0,0,0) \quad\text{ for all } (u_*,v_0)\in B_\rho(0) \subset e^{-\omega}\bbE_u\times W_p^1(\Omega)^n.
\end{align*}
Using that $v_t = -\rho_0^{-1}\nabla (u+u_*)$, we may replace $\bbE_v$ by the smaller space
\begin{align*}
  \{v \in BUC^1(\mathbb{R}_+;W_p^1(\Omega)^n) : e^{\omega t}v_t \in H^{3/2}_p(\mathbb{R}_+;L_p(\Omega)^n)\cap W^1_p(\bbR_+;W_p^1(\Omega)^n))\},
\end{align*}
(which is the same as in \eqref{eq:RegSol}), since the gradient
\begin{align*}
  \nabla : H^{3/2}_p(\bbR_+;W^1_p(\Omega)) \cap W^1_p(\bbR_+;W^2_p(\Omega)) \to H^{3/2}_p(\bbR_+;L_p(\Omega)) \cap W^1_p(\bbR_+;W^1_p(\Omega))
\end{align*}
is continuous and the embedding
\begin{align*}
  \bbE_u \hookrightarrow H^{3/2}_p(\bbR_+;W^1_p(\Omega)) \cap W^1_p(\bbR_+;W^2_p(\Omega))
\end{align*}
is valid by the mixed derivative theorem \cite[Proposition 3.2]{MeSc12}.
This means that the pair $(u+u_*,v):=\psi(u_*,v_0)+(u_*,0)$ solves the main problem \eqref{prob:Kuznetsov} for Kuznetsov's equation, whenever $(u_*,v_0)$ is small enough and $u_*$ satisfies the prescribed boundary condition $u_*|_\Gamma=g$ and initial conditions $u_*|_{t=0}=u_0$, $u_*|_{t=0}=u_1$.
We therefore define $u_*\in e^{-\omega}\bbE_u$ as the unique solution to \eqref{prob:Kuznetsov_Linearized} due to \prettyref{lem:Max_Reg_for_Linearized_Problem} with $(f=0,g,u_0,u_1)$.  We introduce the Banach function spaces
\begin{align}\label{eq:space_Y}
  \widetilde\bbY &:= e^{-\omega}\bbF_g \times W_p^2(\Omega) \times W_p^{2-2/p}(\Omega), \nonumber \\
  \mathbb{Y} &:= \begin{cases}
  \{(g,u_0,u_1)\in\widetilde\bbY : g|_{t=0}=u_0|_{\Gamma}\},\ &\text{if}\ p<3/2,\\
  \{(g,u_0,u_1)\in\widetilde\bbY : g|_{t=0}=u_0|_{\Gamma},\ g_t|_{t=0}=u_1|_{\Gamma}\},\ &\text{if}\ p>3/2,
  \end{cases}
\end{align}
with norm $\norm\cdot_\bbY=\norm\cdot_{\widetilde\bbY}$.  Maximal regularity implies that $u_*$ depends linearly and continuously on $(g,u_0,u_1)\in\bbY$ and thus satisfies the estimate
$$\norm{u_*}_{e^{-\omega}\bbE_u}\lesssim \norm {(g,u_0,u_1)}_\bbY = \|g\|_{e^{-\omega}\bbF_g}+\|u_0\|_{W_p^2}+\|u_1\|_{W_p^{2-2/p}}.$$
Since $\psi$ is analytic on $B_\rho(0)$ w. r. t. $(u_*,v_0)$, 
it follows that $(u,v)\in\0\bbE_{u,\Gamma}\times\bbE_v$ depends analytically on $(g,u_0,u_1)\in \bbY$ and $v_0\in W^1_p(\Omega)$ in a neighborhood of zero.

%

By a basic embedding and by the temporal trace theorem, we obtain
\begin{align*}
  w:=u+u_* &\in e^{-\omega}BUC(\bbR_+;W^2_p(\Omega)), &\quad \norm{w(t)}_{W^2_p} &\leq C e^{-\omega t},\\
  w_t=u_t+(u_*)_t &\in e^{-\omega}BUC(\bbR_+;W^{2-2/p}_p(\Omega)),&\quad \norm{w_t(t)}_{W^{2-2/p}_p} &\leq C e^{-\omega t},
\end{align*}
where $C \geq \max\{ \norm{w}_{e^{-\omega}BUC(\bbR_+;W^2_p)}, \norm{w_t}_{e^{-\omega}BUC(\bbR_+;W^{2-2/p}_p)} \}$.  The representation
\begin{align*}
  e^{\omega t}(v(t) - v_\infty) &=  -\rho_0^{-1}\int_t^\infty e^{\omega t}\nabla w(s)ds = -\rho_0^{-1} \left((e^{-\omega s}\chi_{\bbR_+}(s))\ast(e^{\omega s}\nabla w(s))\right)(t)
\end{align*}
shows that also $v(t)\to v_\infty$ in $W_p^1(\Omega)^n$  and $v_t(t)\to 0$ in $W_p^1(\Omega)^n$ exponentially, as $t\to\infty$.  Furthermore, if $p > 2$, then also $v_{tt}=-\rho_0^{-1}\nabla w_t\to 0$ in $W_p^{1-2/p}(\Omega)^n$ exponentially.  This concludes the proof of Theorem \ref{thm:GLOBWP}.
\end{proof}

\section{Higher regularity}

In this section we establish higher temporal regularity for the solution $(u^*,v^*)$ to the initial-boundary value problem \eqref{prob:Kuznetsov} for Kuznetsov's equation for given boundary data $g$.  We employ the parameter trick of Angenent \cite{Ang90}, where an artificial parameter $\lambda\in(1-\eps,1+\eps)$ is introduced by
\begin{align*}
  u_\lambda(t,x):=u^*(\lambda t,x), \quad v_\lambda(t,x) := v^*(\lambda t,x), \quad g_\lambda(t,x) := g(\lambda t,x).
\end{align*}
It is then rather easy to prove that the equations in \eqref{prob:Kuznetsov} depend $C^k$-differentiably on $\lambda$, provided that $g$ satisfies an appropriate regularity condition \eqref{eq:higher_regularity_for_g}.  If we can establish that this function is also $C^k$ w. r. t. $\lambda$, we obtain for instance
\begin{align*}
  \left.\del_\lambda^j u_\lambda(t,x)\right|_{\lambda=1} = t^j \del_t^j u^*(t,x), \quad \text{for all }j\leq k,
\end{align*}
which implies that $u^*$ gains temporal regularity on every interval $(\delta,\infty)$, $\delta>0$.

We note that the transformation $T_\lambda : f(\cdot,\cdot) \mapsto f(\lambda\cdot,\cdot)$ is a bijection of $\bbX$, where $\bbX$ denotes one of the spaces $L_p(\bbR_+\times\Omega)$, $\bbE_1$, $\bbF_\Gamma$, $\bbE_u$, $\bbF_g$ on $\bbR_+$.  This property follows from identities like $\norm{f(\lambda\cdot)}_{L_p(\bbR_+)} = \lambda^{-1/p}\norm{f}_{L_p(\bbR_+)}$ and $\norm{\del_t(f(\lambda\cdot))}_{L_p(\bbR_+)} = \lambda^{1-1/p}\norm{f}_{L_p(\bbR_+)}$.
However, this is not the case when dealing with exponential weights. Here we obtain for instance $\norm{e^{\lambda\omega \cdot}f(\lambda\cdot)}_{L_p(\bbR_+)} = \lambda^{-1/p} \norm{e^{\omega \cdot}f}_{L_p(\bbR_+)}$ and therefore $T_\lambda : e^{-\omega} \bbX \to e^{-\lambda\omega}\bbX$ is bijective.  This is the reason why we require $\omega_g>\omega$ in \prettyref{thm:regularity_and_decay}, which implies a faster decay of $g$.

\begin{proof}[Proof of \prettyref{thm:regularity_and_decay}]
Let us start with the unique solution $(u^*,v^*)\in e^{-\omega}\bbE_u\times \bbE_v$ of \eqref{prob:Kuznetsov}, which is obtained from \prettyref{thm:GLOBWP}. For $\lambda\in (1-\eps,1+\eps)$ for a sufficiently small $\varepsilon>0$ we define the scaled functions $u_\lambda$, $v_\lambda$, $g_\lambda$ as above. It follows that
\begin{align*}
  \partial_{t}^j u_\lambda(t,x) = \lambda^j \del_t^j u^*(\lambda t,x), \quad \partial_t^2(u_\lambda^2(t,x))=\lambda^2((u^*)^2)_{tt}(\lambda t,x),
\end{align*}
and analogous relations are valid for $v_\lambda$.  This yields that $(u_\lambda,v_\lambda)$ solves the problem
\begin{align}\label{prob:Kuznetsov-scaled}
   \left\{\begin{aligned}
      &\partial_t^2u_{\lambda} - \lambda^2c^2\Delta u_\lambda - \lambda b\Delta \partial_t u_\lambda \\
      &= k\partial_t^2(u_\lambda^2)+2\lambda(v_\lambda\cdot \nabla \partial_t u_\lambda)+2\rho_0^{-1}\lambda^2(\nabla u_\lambda)^2 &&\quad\text{in }J\times\Omega,\\
       \partial_t v_\lambda&=-\lambda\rho_0^{-1}\nabla u_\lambda &&\quad\text{in }J\times\Omega,\\
      u_\lambda|_{\Gamma} &= g_\lambda &&\quad\text{in }J\times\Gamma,\\
      u_\lambda(0) &= u_0 &&\quad\text{in }\Omega,\\
      \partial_t u_\lambda(0) &= \lambda u_1 &&\quad\text{in }\Omega,\\
      v_\lambda(0) &= v_0 &&\quad\text{in }\Omega.
   \end{aligned}\right.
\end{align}


\begin{lem}\label{lem:Nonlinearity_for_higher_regularity_is_analytic}
  Let 
  $p\in(1,\infty)$, $p>\max\{1,n/2\}$, $p\neq 3/2$ and let $\omega\in(0,\omega_0)$ have the same meaning as in \prettyref{lem:Max_Reg_for_Linearized_Problem}.  Let $(g,u_0,u_1)\in\bbY$ (defined in \eqref{eq:space_Y}), $v_0\in W^1_p(\Omega)^n$ and suppose that there exists $\omega_g>\omega$ such that
  \begin{align}\label{eq:higher_regularity_for_g}
    g, \, [t\mapsto t \del_t g(t)], \,\ldots\,, [t\mapsto t^k \del_t^k g(t)] \in e^{-\omega_g}\bbF_g \quad\text{for some }k\geq 1.
  \end{align}
  Let $\bbE_u$, $\bbE_v$, $\bbE_{v_t}$ be the same spaces as in \eqref{eq:RegSol_2}.  Then the mapping
  \begin{align*}
    H:(1-\varepsilon,1+\varepsilon)\times e^{-\omega}\bbE_u\times \bbE_v \to e^{-\omega}L_p(\mathbb{R}_+\times\Omega)\times \bbE_{v_t} \times\mathbb{Y}\times W_p^1(\Omega)^n,
  \end{align*}
  defined by
  \begin{align*}
    &H(\lambda,u,v) :=\begin{pmatrix}\partial_t^2u - \lambda^2c^2\Delta u - \lambda b\Delta \partial_t u - k\partial_t^2(u^2)-2\lambda v{\cdot} \nabla \partial_t u-2\rho_0^{-1}\lambda^2(\nabla u)^2\\
\partial_t v+\lambda \rho_0^{-1}\nabla u\\
u|_{\Gamma}-g_\lambda\\
u(0)-u_0\\
u_t(0)-\lambda u_1\\
v(0)-v_0
\end{pmatrix}.
  \end{align*}
  is $C^k$ and there exists $\rho>0$ such that the first Fr\'{e}chet derivative of $H$ w. r. t. $(u,v)$ at $(1,u^*,v^*)$ is invertible, provided that
  \begin{align*}
    \norm{(g,u_0,u_1)}_\bbY + \norm {v_0}_{W^1_p(\Omega)} < \rho.
  \end{align*}
\end{lem}

\begin{proof}
For every $\lambda$, the mapping $(u,v)\mapsto H(\lambda,u,v)$ is analytic by \prettyref{lem:Nonlinearity_is_analytic}.  To obtain the differentiability w. r. t. $\lambda$, we compute
$$\del_\lambda H(\lambda,u,v):=\begin{pmatrix}-2\lambda c^2\Delta u-b\Delta \partial_t u - 2v\cdot \nabla \partial_t u-4\rho_0^{-1}\lambda(\nabla u)^2\\
\rho_0^{-1}\nabla u\\
-\del_\lambda g_\lambda\\
0\\
-u_1\\
0
\end{pmatrix}.$$
From this formula we infer that $\lambda\mapsto \del^j_\lambda H(\lambda,u,v)$, $j\leq k$, is well-defined and continuous, since $g$ satisfies \eqref{eq:higher_regularity_for_g} and $\del_\lambda^j g_\lambda(t,x) = t^j \del_t^j g(\lambda t,x)$ and the functions $(\del_\lambda g_\lambda,0,u_1)$ satisfy the relevant compatibility conditions in the definition of $\bbY$, since
\begin{align*}
  \del_\lambda g_\lambda|_{t=0}=[tg_t(\lambda t)]|_{t=0}=0, \quad (\del_\lambda g_\lambda)_t|_{t=0}=g_t(\lambda t)|_{t=0}+[t g_{tt}(\lambda t)]|_{t=0}=u_1|_{\partial\Omega}.
\end{align*}
The derivative of $H$ w.r.t\ $(u,v)$ at $(1,u^*,v^*)$ reads as follows
\begin{multline*}
D_{(u,v)}H(1,u^*,v^*)[\bar{u},\bar{v}]=\\
=\begin{pmatrix}\partial_t^2\bar{u} - c^2\Delta \bar{u} - b\Delta \partial_t \bar{u} - 2k\partial_t^2(u^*\bar{u})-2\bar{v}{\cdot} \nabla \partial_t u^*-2v^*{\cdot} \nabla \partial_t \bar{u}-2\rho_0^{-1}\nabla u^*{\cdot}\nabla\bar{u}\\
\partial_t \bar{v}+ \rho_0^{-1}\nabla\bar{u}\\
\bar{u}|_{\partial\Omega}\\
\bar{u}(0)\\
\bar{u}_t(0)\\
\bar{v}(0)
\end{pmatrix}.
\end{multline*}
The fact that $D_{(u,v)}H(1,u^*,v^*)$ is an isomorphism 
follows from a Neumann series argument. Indeed, if the norms of the data $(g,u_0,u_1,v_0)\in\mathbb{Y}\times W_p^1(\Omega)^n$ are sufficiently small, then the coefficients of the terms involving $(u^*,v^*)$ in the first component are small as well, since $(u^*,v^*)\in e^{-\omega}\mathbb{E}_u\times\mathbb{E}_v$ depends continuously on the data $(g,u_0,u_1,v_0)\in\mathbb{Y}\times W_p^1(\Omega)^n$.
\end{proof}

Since $H(1,u^*,v^*)=(0,0,0,0,0,0)$, the implicit function theorem yields a (possibly) small number $\rho\in (0,\varepsilon)$ and a mapping $\psi\in C^k((1-\rho,1+\rho);e^{-\omega}\bbE_u\times \bbE_v)$ with $\psi(1)=(u^*,v^*)$ such that $H(\lambda,\psi(\lambda))=(0,0,0,0,0,0)$ for each $\lambda\in (1-\rho,1+\rho)$. Here again $\bbE_v$ may be taken to be the space in \eqref{eq:RegSol}.  By uniqueness it follows that $(u_\lambda,v_\lambda)=\psi(\lambda)$ for each $\lambda\in (1-\rho,1+\rho)$ and therefore
$$\Big[t\mapsto (t^j \partial_t^{j}u^*(t),t^{j}\partial_t^{j}v^*(t))=[\psi^{(j)}(1)](t)\Big]\in e^{-\omega}\bbE_u\times \bbE_v,$$
hence
$$(\partial_t^{j}u^*,\partial_t^{j}v^*)\in e^{-\omega}\bbE_u([\delta,\infty))\times \bbE_v([\delta,\infty)),$$
for each $\delta>0$ and $j\in\{0,\ldots, k\}$. In particular this yields that $\partial_t^{j}u^*(t)\to 0$ in $W_p^{2}(\Omega)$ for $j\in\{0,\ldots,k\}$, $\partial_t^{(k+1)}u^*(t)\to 0$ in $W_p^{2-2/p}(\Omega)$, and $\partial_t^{j}v^*(t)\to 0$ in $W_p^1(\Omega)^n$ for $j\in\{1,\ldots,k+1\}$ as $t\to\infty$ at the exponential rate $\omega>0$. The proof of \prettyref{thm:regularity_and_decay} is complete.
\end{proof}

\begin{proof}[Proof of \prettyref{cor:g_finite_time_interval}]
  Using \prettyref{thm:GLOBWP}, we solve \eqref{prob:Kuznetsov} on $\bbR_+$.  By the temporal trace theorem we see that $u(T)\in W^2_p(\Omega)$, $u_t(T)\in W^{2-2/p}_p(\Omega)$ and $v(T)\in W^1_p(\Omega)^n$ depend continuously on $(u,v)$ and thus continuously on the data $(g,u_0,u_1,v_0)$.  We choose $\tilde\rho\leq \rho$ sufficiently small such that
  \begin{align*}
    \norm{u(T)}_{W^2_p}+\norm{u_t(T)}_{W^{2-2/p}_p}+\norm{v(T)}_{W^1_p}\leq \rho.
  \end{align*}
  Using a translation $u(\cdot) \mapsto u(\cdot-T)$ and applying \prettyref{thm:regularity_and_decay} we obtain the assertion.
\end{proof}
\bibliographystyle{plain}
\bibliography{MeWi11_Lit}

\begin{thebibliography}{10}

\bibitem{AdFo03}
Robert~A. Adams and John J.~F. Fournier.
\newblock {\em Sobolev spaces}, volume 140 of {\em Pure and Applied Mathematics
  (Amsterdam)}.
\newblock Elsevier/Academic Press, Amsterdam, second edition, 2003.

\bibitem{Ama93}
Herbert Amann.
\newblock Nonhomogeneous linear and quasilinear elliptic and parabolic boundary
  value problems.
\newblock In {\em Function spaces, differential operators and nonlinear
  analysis ({F}riedrichroda, 1992)}, volume 133 of {\em Teubner-Texte Math.},
  pages 9--126. Teubner, Stuttgart, 1993.

\bibitem{Ama95}
Herbert Amann.
\newblock {\em Linear and quasilinear parabolic problems. {V}ol. {I}:
  {A}bstract linear theory.}, volume~89 of {\em Monographs in Mathematics}.
\newblock Birkh\"auser Boston Inc., Boston, MA, 1995.

\bibitem{Ang90}
Sigurd~B. Angenent.
\newblock Nonlinear analytic semiflows.
\newblock {\em Proc. Roy. Soc. Edinburgh Sect. A}, 115(1-2):91--107, 1990.

\bibitem{Dei85}
Klaus Deimling.
\newblock {\em Nonlinear {F}unctional {A}nalysis}.
\newblock Springer-Verlag, Berlin, 1985.

\bibitem{DSS08}
R.~Denk, J.~Saal, and J.~Seiler.
\newblock Inhomogeneous symbols, the {N}ewton polygon, and maximal
  {$L^p$}-regularity.
\newblock {\em Russ. J. Math. Phys.}, 15(2):171--191, 2008.

\bibitem{DHP03}
Robert Denk, Matthias Hieber, and Jan Pr\"{u}ss.
\newblock {$\mathcal R$}-boundedness, {F}ourier multipliers and problems of
  elliptic and parabolic type.
\newblock {\em Mem. Amer. Math. Soc.}, 166(788):viii+114, 2003.

\bibitem{DHP07}
Robert Denk, Matthias Hieber, and Jan Pr\"{u}ss.
\newblock Optimal ${L}^{p}$- ${L}^{q}$-estimates for parabolic boundary value
  problems with inhomogeneous data.
\newblock {\em Math. Z.}, 257(1):193--224, 2007.

\bibitem{diB84}
Gabriella Di~Blasio.
\newblock Linear parabolic evolution equations in {$L^p$}-spaces.
\newblock {\em Ann. Mat. Pura Appl. (4)}, 138:55--104, 1984.

\bibitem{KaLa12}
Barbara Kaltenbacher and Irena Lasiecka.
\newblock An analysis of nonhomogeneous {K}uznetsov's equation: local and
  global well-posedness; exponential decay.
\newblock {\em Math. Nachr.}, 285(2-3):295--321, 2012.

\bibitem{Kal07}
Manfred Kaltenbacher.
\newblock {\em Numerical simulation of mechatronic sensors and actuators}.
\newblock Springer, 2007.

\bibitem{Kuz71}
V.~P. Kuznetsov.
\newblock Equations of nonlinear acoustics.
\newblock {\em Sov. Phys. Acoust.}, 16:467--470, 1971.

\bibitem{LatPrSchn06}
Yuri Latushkin, Jan Pr\"{u}ss, and Roland Schnaubelt.
\newblock Stable and unstable manifolds for quasilinear parabolic systems with
  fully nonlinear boundary conditions.
\newblock {\em J. Evol. Equ.}, 6(4):537--576, 2006.

\bibitem{MeyWi11}
Stefan Meyer and Mathias Wilke.
\newblock Optimal regularity and long-time behavior of solutions for the
  {W}estervelt equation.
\newblock {\em Appl. Math. Optim.}, 64(2):257--271, 2011.

\bibitem{MeSc12}
Martin Meyries and Roland Schnaubelt.
\newblock Interpolation, embeddings and traces of anisotropic fractional
  {S}obolev spaces with temporal weights.
\newblock {\em J. Funct. Anal.}, 262(3):1200--1229, 2012.

\bibitem{Tri83}
Hans Triebel.
\newblock {\em Theory of function spaces}, volume~78 of {\em Monographs in
  Mathematics}.
\newblock Birkh\"auser Verlag, Basel, 1983.

\bibitem{Tri95}
Hans Triebel.
\newblock {\em Interpolation theory, function spaces, differential operators}.
\newblock Johann Ambrosius Barth, Heidelberg, second edition, 1995.

\bibitem{Zei86I}
Eberhard Zeidler.
\newblock {\em Nonlinear {F}unctional {A}nalysis and its {A}pplications. {I}}.
\newblock Springer-Verlag, New York, 1986.
\newblock Fixed-point theorems, Translated from the German by Peter R. Wadsack.

\end{thebibliography}

\end{document}